\documentclass[10pt, a4paper]{amsart}
\usepackage[arXiv]{rseaineng}
\usepackage[cal]{rseaincmds}
\usepackage{hhline, array, bm}
\usepackage{ifthen}
\usepackage{nicematrix}

\usepackage{cellspace}
\allowdisplaybreaks

\DeclareMathOperator{\height}{ht}

\numberwithin{equation}{section}

\title[]{Equivariant Ehrhart series for rational simplices and rational cross-polytopes}
\author[Ryo Uchiumi]{Ryo Uchiumi}
\date{\today}
\subjclass[2020]{52B20, 05E18, 20C10} %17B22 : root system
\keywords{}
\thanks{The author was supported by JSPS KAKENHI, Grant Number 25KJ1735}

\begin{document}

\begin{abstract}
	Equivariant Ehrhart theory was introduced by Stapledon (2011) and has since been extensively studied. 
	A central question is whether the numerator in a certain rational function expression for the equivariant Ehrhart series is both a polynomial and effective.
	
	In this paper, we propose a modification of the denominator in the rational function expression for the equivariant Ehrhart series of rational polytopes. For rational simplices and rational cross-polytopes whose vertices have coordinates that are reciprocals of integers, we show that the resulting numerator is a polynomial and is effective.
\end{abstract}

\maketitle

\tableofcontents

%%%%%%%%%%%%%%%%%%%%%%%%%%%%%%%%%%%%%%%%%%%%%%%%%%%%%%%

%\fontsize{10pt}{15pt}\selectfont

\section{Introduction}

Let $L \cong \mathbb{Z}^\ell$ be a lattice, and $L_\mathbb{K} \ceq L \otimes \mathbb{K} \cong \mathbb{K}^\ell$ for $\mathbb{K} = \mathbb{Q}$ or $\mathbb{R}$.
Let $P \subseteq L_\mathbb{R}$ be an $\ell$-dimensional polytope.
The \textbf{Ehrhart function} $\Ell_P(q) = \#(qP \cap L)$ with $q \in \mathbb{Z}_{\geq 0}$
is a counting function of lattice points of $qP$.
If $P$ is a lattice polytope, then $\Ell_P(q)$ is a polynomial in $q$, called the \textbf{Ehrhart polynomial} of $P$.
If $P$ is a rational polytope, then $\Ell_P(q)$ is a quasi-polynomial in $q$, that is, there exist $\tilde{n} \in \mathbb{Z}_{>0}$ and polynomials $f^{(1)}(t),\ldots,f^{(\tilde{n})}(t) \in \mathbb{Q}[t]$ such that
\begin{align}
	\Ell_P(q) = f^{(r)}(q) \IF q \equiv r \pmod{\tilde{n}}.
\end{align}
Then $\Ell_P(q)$ is called the \textbf{Ehrhart quasi-polynomial} of $P$.
The number $\tilde{n}$ is called a \textbf{period} and the polynomial $f^{(r)}(t)$ is called the \textbf{$\bm r$-th constituent}.
The denominator of $P$, the smallest positive integer $k$ for which $kP$ is a lattice polytope, is a period of the quasi-polynomial $\Ell_P(q)$, but it is not necessarily the minimum period.
Note that a polynomial is a quasi-polynomial of period $1$.

Let $P^\circ$ denote the interior of $P$, and define $\Ell_{P^\circ}(q) = \#(qP^\circ \cap L)$.
Then the following formula is known as \textbf{Ehrhart reciprocity}:
\begin{align}
	\Ell_P(-q) = (-1)^\ell \Ell_{P^\circ}(q).
\end{align}

The \textbf{Ehrhart series} of $P$ is a power series given by
\begin{align}
	\Ehr(P;t) \ceq \sum_{q \in \mathbb{Z}_{\geq 0}}\Ell_P(q)t^q.
\end{align}
If $P$ is a rational polytope, then $\Ehr(P;t)$ is a rational function of $t$.
In particular, it is well known that it has the form
\begin{align}
	 \Ehr(P;t) = \frac{h^\ast(t)}{(1-t^d)^{\ell+1}} \label{series}
\end{align}
for some polynomial $h^\ast(t) \in \mathbb{Z}[t]$ called the \textbf{$\bm h^{\bm\ast}$-polynomial} of $P$, where $d$ is the denominator of $P$.
Following Stanley \cite[\S1]{StanleyDecomp}, %Decompositions of rational convex polytopes
each coefficient of $h^\ast(t)$ is non-negative.

Stapledon introduced the equivariant Ehrhart theory, which is a generalization of Ehrhart theory, in \cite{Stapledon}.
This is closely related to toric geometry, representation theory, and mirror symmetry.
Let $\GAMMA$ be a finite group and $\Cl(\GAMMA)$ denote the set of class functions on $\GAMMA$.
Suppose that $\GAMMA$ acts linearly on a lattice $L$ via $\rho : \GAMMA \lra \GL(L)$.
Let $L' \ceq L \oplus \mathbb{Z} \cong \mathbb{Z}^{\ell+1}$ and $\rho' : \GAMMA\lra \GL(L')$ be the representation induced by 
\begin{align}
	\rho'(\gamma) : (x,n) \lmapsto (\rho(\gamma)(x),\,n)
	\quad \tfor \gamma \in \GAMMA,\ (x,n) \in L'.
\end{align}
Let $P \subseteq L_\mathbb{R}$ be an $\ell$-dimensional rational polytope, and suppose that $P$ is $\GAMMA$-invariant.
Then we can consider the character of the permutation representation of $\GAMMA$ on $qP \cap L$ for each $q \in \mathbb{Z}_{>0}$, denoted by $\chi_{P,q}$.
Evaluating $\chi_{P,q}$ at the identity element $1 \in \GAMMA$ recovers the Ehrhart quasi-polynomial:
\begin{align}
	\chi_{P,q}(1) = \Ell_P(q).
\end{align}
Thus it is the equivariant analogue of the Ehrhart quasi-polynomial.
Furthermore, a function $q \lmapsto \chi_{P,q}$ is a quasi-polynomial in $q$, and the coefficients of each of its constituents are linear combinations of irreducible characters of $\GAMMA$ with rational coefficients.

Let $\chi_{P,q}^\ast$ denote the character of the permutation representation of $\GAMMA$ on $qP^\circ \cap L$.
Then the following formula is the equivariant version of the Ehrhart reciprocity:
\begin{align}
	\chi_{P,-q} = (-1)^\ell \cdot \det{\rho} \cdot \chi_{P,q}^\ast. \label{recip}
\end{align}

The \textbf{equivariant Ehrhart series} is a power series given by
\begin{align}
	\Ehr_\rho(P;t) = \sum_{q \in \mathbb{Z}_{\geq 0}}\chi_{P,q} t^q.
\end{align}
In the case where $P$ is a lattice polytope, expressions of the form 
\begin{align}
	\Ehr_\rho(P;t) = \frac{H^\ast[t]}{(1-t)\det(I - t\rho )} = \frac{H^\ast[t]}{\det(I - t\rho' )} \label{equivseries}
\end{align}
for some $H^\ast[t] = \sum_{n \in \mathbb{Z}_{\geq 0}}H_n^\ast t^n$ have been well studied extensively, where $\{H_n^\ast\}_{n \in \mathbb{Z}_{\geq 0}}$ are virtual characters of $\GAMMA$, that is, each $H_n^\ast$ is a linear combination of irreducible characters of $\GAMMA$ with integer coefficients.
Evaluating $H^\ast[t]$ at the identity element recovers the $h^\ast$-polynomial.
Hence the series $H^\ast[t]$ is the equivariant analogue of the $h^\ast$-polynomial.
Note that $H^\ast[t]$ may not always be a polynomial (see \cite[Lemma 7.3 and Example 7.5]{Stapledon}).
We say that $H^\ast[t]$ is \textbf{effective} if each $H_n^\ast$ is a character, that is, if $H_n^\ast$ is a linear combination of irreducible characters of $\GAMMA$ with non-negative integer coefficients.
If $H^\ast[t]$ is effective, then it is a polynomial.
The converse has been studied as the effectiveness conjecture {\cite[Conjecture 12.1]{Stapledon}}.
Very recently, several counterexamples to the effectiveness conjecture have been independently given by Chalmers \cite[Theorem]{Chalmers}, Clarke \cite[Example 5.6 and Example 5.7]{Clarke}, and Ziegler \cite[Theorem 1.1 and Theorem 1.2]{Ziegler}.
Earlier, Clarke--Higashitani--K\"olbl had already given an example, in a setting in which $P$ is allowed to be a rational polytope, in which $H^\ast[t]$ is a polynomial but is not effective \cite[Example 4.4]{ClarkeHigashitaniKolbl}.

In this paper, we propose a modification of the denominator used in the rational function expression for the equivariant Ehrhart series of rational polytopes.
More specifically, we consider the form
\begin{align}
	\Ehr_\rho(P;t) = \frac{\varphi[t]}{(1-t^{d_0})\det(I - TR)} \OR  
	\Ehr_\rho(P;t) = \frac{\varphi[t]}{\det(I - T'R')} \label{EEhr}
\end{align}
for some $\varphi[t] = \sum_{n \in \mathbb{Z}_{\geq 0}}\varphi_nt^n \in \Cl(\GAMMA)[[t]]$, where $T  =\diag(t^{d_1},\ldots,t^{d_\ell})$ and $T' = \diag(t^{d_0},t^{d_1},\ldots,t^{d_\ell})$ for given $d_0,\ldots,d_\ell \in \mathbb{Z}_{>0}$, and $R : \GAMMA\lra \GL(\mathbb{Z}^\ell)$ and $R' : \GAMMA \lra \GL(\mathbb{Z}^{\ell+1})$ are homomorphisms associated with $\rho$ and $\rho'$ with respect to a fixed basis, respectively.
Note that each $\varphi_n$ need not be a virtual character for certain choices of integers and homomorphisms.
In \cref{sec1}, we will investigate which representation gives rise to the denominator.
In \cref{sec2}, we will show that, for rational simplices, there exist appropriate choices of $d_0,\ldots,d_\ell$ and $R'$ such that $\varphi[t]$ is a polynomial and effective (\cref{thm2.1} and \cref{cor2.3}).
In \cref{sec3}, we will show that, for cross-polytopes whose vertices have coordinates that are reciprocals of integers, there exist appropriate choices of $d_0,\ldots,d_\ell$ and $R$ such that $\varphi[t]$ is a polynomial and effective (\cref{thm3.3} and \cref{thm4.6}).
Furthermore, such $\varphi[t]$ is palindromic.
This class of polytopes contains the counterexample found by Clarke--Higashitani--K\"olbl.
In \cref{sec3.1}, we will review the hyperoctahedral group, the automorphism group of a cross-polytope.

\section{Series of graded polynomial rings}\label{sec1}

Fix $d_1,\ldots,d_\ell \in \mathbb{Z}_{>0}$, and let $T \ceq \diag(t^{d_1},\ldots,t^{d_\ell})$.
Let $\mathbb{C}[z] \ceq \mathbb{C}[z_1,\ldots,z_\ell]$ be the graded $\mathbb{C}$-algebra with $\deg{z_i} = d_i$ for $i \in \{1,\ldots,\ell\}$, and let $\mathbb{C}[z]_m$ denote the linear subspace of $\mathbb{C}[z]$ generated by homogeneous monomials of degree $m$.

Let $\GAMMA$ be a finite group and $R : \GAMMA\lra \GL(\mathbb{Z}^\ell)$ a group homomorphism.
Then $R$ induces a $\GAMMA$-action on $\mathbb{C}[z]$ defined as 
\begin{align}
	\gamma : c_1z_1 + \cdots + c_\ell z_\ell = (c_1 \ \cdots \ c_\ell)\begin{pmatrix}
		z_1\\ \vdots \\ z_\ell
	\end{pmatrix} \lmapsto (c_1 \ \cdots \ c_\ell)R(\gamma)\begin{pmatrix}
	z_1\\ \vdots \\ z_\ell
	\end{pmatrix}.
\end{align}

\begin{lemma}\label{lem1.1}
	Let $R : \GAMMA\lra \GL(\mathbb{Z}^\ell)$ be a group homomorphism.
	Suppose that, for each $\gamma \in \GAMMA$, there exists an invertible matrix $P_\gamma$ such that $P_\gamma R(\gamma)P_\gamma^{-1}$ is a diagonal matrix and $P_\gamma$ commutes with $T$.
	Then $\mathbb{C}[z]_m$ is a representation space of $\GAMMA$ for each $m \in \mathbb{Z}_{\geq 0}$.
	Furthermore, we have
	\begin{align}
		\sum_{m \in \mathbb{Z}_{\geq 0}} \chi_{\mathbb{C}[z]_m} t^m = \frac{1}{\det(I - TR)},
	\end{align}
	where $\chi_{\mathbb{C}[z]_m}$ is the character of $\GAMMA$ on $\mathbb{C}[z]_m$.
\end{lemma}
\begin{proof}
	Let $\gamma \in \GAMMA$, and $\{u_1,\ldots,u_\ell\}$ be a basis for $\mathbb{C}[z]$ obtained from $\{z_1,\ldots,z_\ell\}$ via $P_\gamma$, that is,
	\begin{align}
		u_i = \sum_{j=1}^\ell p_{ij}z_j \quad \tfor i \in \{1,\ldots,\ell\},
	\end{align}
	where $P_\gamma = (p_{ij})_{ij}$.
	Since $P_\gamma$ and $T$ commute, $p_{ij}  \neq 0$ implies $d_i = d_j$. 
	Hence each $u_i$ is a homogeneous element of $\mathbb{C}[z]$ with $\deg{u_i} = d_i$.
	Furthermore, we may regard $\gamma$ as acting via the diagonal matrix $P_\gamma R(\gamma) P_\gamma^{-1} = \diag(\lambda_1,\ldots,\lambda_\ell)$.
	Therefore, the action of $\GAMMA$ on $\mathbb{C}[z]$ preserves the degree, that is,  $\mathbb{C}[z]_m$ is a representation space for each $m \in \mathbb{Z}_{\geq 0}$.
	Hence we have
	\begin{align}
			\sum_{m \in \mathbb{Z}_{\geq 0}}\chi_{\mathbb{C}[z]_m}(\gamma) t^m
			&= \sum_{m \in \mathbb{Z}_{\geq 0}} \sum_{\substack{c \in \mathbb{Z}_{\geq 0}^{\ell}\\ \sum_{i=1}^\ell c_id_i = m}} \left( \prod_{i = 1}^\ell \lambda_i^{c_i} \right) t^m\\
			&= \sum_{c \in \mathbb{Z}_{\geq 0}}\left( \prod_{i=1}^\ell \lambda_i^{c_i} \right) t^{\sum_{i=1}^\ell c_id_i}\\
			&= \prod_{i=1}^\ell \left( \sum_{n \in \mathbb{Z}_{\geq 0}}(\lambda_i t^{d_i})^n \right)\\
			&= \frac{1}{\prod_{i=1}^\ell(1-t^{d_i}\lambda_i)}\\
			&= \frac{1}{\det(I - T P_\gamma R(\gamma) P_\gamma^{-1})}\\
			&= \frac{1}{\det{P_\gamma} \cdot \det(I - T R(\gamma)) \cdot \det{P_\gamma^{-1}}}\\
			&= \frac{1}{\det(I - T R(\gamma))}. 	\label{s3} \qedend
	\end{align}
\end{proof}

\section{Simplices} \label{sec2}

\noindent

Let $L$ be a lattice of rank $\ell$, and let $P \subseteq L_\mathbb{R}$ be a rational $\ell$-simplex with vertex set $V_P = \{v_0,\ldots,v_\ell\} \subseteq L_\mathbb{Q}$.
Suppose that $d = (d_0,\ldots,d_\ell)$ are positive integers such that $d_0v_0, \ldots, d_\ell v_\ell \in L$.

Let $L' \ceq L \oplus \mathbb{Z}$, and $L'_\mathbb{Q} \ceq L' \otimes \mathbb{Q}$.
We define the \textbf{height} of $x' = (x,n) \in L'$ to be the integer $n$, and denote it by $\height(x')$.
For $i \in \{0,\ldots,\ell\}$, define $v_i' \ceq (v_i,1) \in L_\mathbb{Q}'$.
Define sets $A(P,d)$ and $B(P,d)$ by
\begin{align}
	A(P,d) &\ceq \bigset{a_0(d_0v_0') + \cdots + a_\ell(d_\ell v_\ell') \in L'}{a_0,\ldots,a_\ell \in \mathbb{Z}_{\geq 0}},\\
	B(P,d) &\ceq \bigset{b_0(d_0v_0') + \cdots + b_\ell (d_\ell v_\ell') \in L'}{0 \leq b_i < 1 \tfor i \in\{0,\ldots,\ell\}}\\
	&= \bigset{b_0v_0' + \cdots + b_\ell v_\ell' \in L'}{0 \leq b_i < d_i \tfor i \in\{0,\ldots,\ell\}},
\end{align}
and let 
\begin{align}
	A(P,d)_m &\ceq \bigset{a \in A(P,d)}{\height(a) = m},\\
	B(P,d)_m &\ceq \bigset{b \in B(P,d)}{\height(b) = m}
\end{align}
for each $m \in \mathbb{Z}_{\geq 0}$.
Note that the set $B(P,d)$ is finite.

We review some results on the Ehrhart series of simplices.
\begin{theorem}[see {\cite[Theorem 1.3]{StanleyDecomp}}] \label{thmStanley}
	Suppose that $d = (d_0,\ldots,d_\ell)$ satisfies
	\begin{align}
		d_i = \min\bigset{k \in \mathbb{Z}_{>0}}{kv_i \in L} \quad \tforall i \in \{0,\ldots,\ell\}. \label{cond_denom}
	\end{align}
	Then we have
	\begin{align}
		\Ehr(P;t) 
		&= \frac{\sum_{b \in B(P,d)}t^{\height(b)}}{(1-t^{d_0})(1 - t^{d_1}) \cdots (1-t^{d_\ell})}\\[5pt]
		&= \frac{\sum_{m \in \mathbb{Z}_{\geq 0}}\#B(P,d)_m t^{m}}{(1-t^{d_0})(1 - t^{d_1}) \cdots (1-t^{d_\ell})}. \label{s31}
	\end{align}
\end{theorem}

\begin{remark}
	Although the notation is changed, the above equation is the same as the first equation in Theorem 1.3 in \cite{StanleyDecomp}.
	Stanley's $\gamma_i$ and $\DELTA_H \cap \mathbb{Z}^{m+1}$ correspond, respectively, to $d_iv_i'$ and $B(P,d)$ in our notation.
\end{remark}

Let $\GAMMA$ be a finite group acting linearly on $L$ via $\rho : \GAMMA \lra \GL(L)$.
Suppose that a rational $\ell$-simplex $P$ is $\GAMMA$-invariant.
Since $\GAMMA$ acts on the vertex set $V_P$, for each $\gamma \in \GAMMA$, we can find a permutation $\sigma_\gamma \in \mathfrak{S}_{\ell+1}$ such that
\begin{align}
	\rho(\gamma)(v_i) = v_{\sigma_\gamma(i)} \quad \tforall i \in \{0,\ldots,\ell\}.
\end{align}
Moreover, suppose that $d_i = d_{\sigma_\gamma(i)}$ for all $\gamma \in \GAMMA$ and $i \in \{0,\ldots,\ell\}$.
Note that, if the condition \cref{cond_denom} is satisfied, then this assumption holds.
Let $R_V' : \GAMMA\lra \GL(\mathbb{Z}^{\ell+1})$ be a homomorphism such that $R_V'(\gamma)$ is the permutation matrix of $\sigma_\gamma$.
Furthermore, $A(P,d)_m$ and $B(P,d)_m$ are invariant under the action $\rho' : \GAMMA\lra \GL(L')$.

Let $T' \ceq \diag(t^{d_0},t^{d_1},\ldots,t^{d_\ell})$.

\begin{theorem}\label{thm2.1}
	With the notation above,
	\begin{align}
		\Ehr_\rho(P;t) = \frac{\sum_{m \in \mathbb{Z}_{\geq 0}}\chi_{B(P,d)_m} t^m}{\det(I - T'R_V')}.	\label{sthm2.1}
	\end{align}
	In particular, the numerator on the right-hand side is a polynomial and effective.
\end{theorem}
If $d_0,\ldots,d_\ell$ satisfy the condition \cref{cond_denom}, then evaluating the above equation at the identity element $1 \in \GAMMA$ recovers \cref{s31}.
Hence \cref{thm2.1} is an equivariant analogue of \cref{thmStanley}.

\begin{proof}[Proof of \cref{thm2.1}]
	Let 
	\begin{align}
		C \ceq \bigset{c_0v_0' + \cdots + c_\ell v_\ell' \in L'}{c_0,\ldots,c_\ell \in \mathbb{R}_{\geq 0}},
	\end{align}
	and
	\begin{align}
		C_m \ceq \bigset{c \in C}{\height(c) = m}
	\end{align}
	for each $m \in \mathbb{Z}_{\geq 0}$.
	It is clear that $C$ is the set of sums of an element of $A(P,d)$ and an element of $B(P,d)$:
	\begin{align}
		C = \bigset{a + b}{a \in A(P,d),\ b \in B(P,d)}.
	\end{align}
	Therefore we can see that
	\begin{align}
		\Ehr_\rho(P;t)
		 = \sum_{m \in \mathbb{Z}_{\geq 0}}\chi_{C_m}t^m 
		 = \left( 
			\sum_{m \in \mathbb{Z}_{\geq 0}}\chi_{A(P,d)_m}t^m
			 \right)\left(
			\sum_{m \in \mathbb{Z}_{\geq 0}}\chi_{B(P,d)_m}t^m 
			 \right).	\label{s1}
	\end{align}

	Let $\mathbb{C}A(P,d)$ be the $\mathbb{C}$-linear space generated by $A(P,d)$, with the grading defined by $\deg{a} = \height{a}$ for $a \in A(P,d)$.
	Each subspace $\mathbb{C}A(P,d)_m$ is a $\GAMMA$-representation and $\chi_{A(P,d)_m}$ is its character.
	Let $z_0,\ldots,z_\ell$ be variables with degrees $\deg(z_i) = d_i$ for each $i \in \{0,\ldots,\ell\}$, and let $\mathbb{C}[z] \ceq \mathbb{C}[z_0,\ldots,z_\ell]$ be a graded $\mathbb{C}$-algebra.
	Then the homomorphism $R_V' : \GAMMA \lra \GL(\mathbb{Z}^{\ell+1})$ defines a representation on $\mathbb{C}[z]$.
	Furthermore, a map $\phi : (\mathbb{C}A(P,d),\, \rho') \lra (\mathbb{C}[z],R_V')$ defined by
	\begin{align}
		\phi : a_0d_0v_0' + \cdots + a_\ell d_\ell v_\ell' \lmapsto z_0^{a_0} \cdots z_\ell^{a_\ell}
	\end{align}
	is an isomorphism of $\GAMMA$-representations that preserves the degree.
	Therefore, by \cref{lem1.1}, we have
	\begin{align}
		\sum_{m \in \mathbb{Z}_{\geq 0}}\chi_{A(P,d)_m} t^m = \sum_{m \in \mathbb{Z}_{\geq 0}}\chi_{\mathbb{C}[z]_m} t^m = \frac{1}{\det(I - T'R'_V)}.	\label{s2}
	\end{align}
	Hence \cref{s1} and \cref{s2} imply \cref{sthm2.1}.
\end{proof}

Let $d(P)$ be a positive integer defined by
\begin{align}
	d(P) = \min\bigset{k \in \mathbb{Z}_{>0}}{kv_0,\ldots,kv_\ell \in L}.
\end{align}
\begin{corollary}\label{cor2.3}
	If $d = (d(P), \ldots, d(P))$, then 
	\begin{align}
		\Ehr_\rho(P;t) = \frac{\sum_{m \in \mathbb{Z}_{\geq 0}}\chi_{B(P,d)_m} t^m}{(1-t^{d(P)})\det(I - t^{d(P)}\rho )}.
	\end{align}
\end{corollary}
Evaluating the above equation at the identity element recovers the form \cref{series}.

\begin{proof}[Proof of \cref{cor2.3}]
	Let $E = \{e_1,\ldots,e_\ell\}$ be a $\mathbb{Z}$-basis for $L$, and define a $\mathbb{Z}$-basis $E'$ for $L'$ as
	\begin{align}
		E' = \{(e_1,1),\, \ldots,\, (e_\ell,1),\, (0,1)\}.
	\end{align}
	Let $R_E : \GAMMA\lra \GL(\mathbb{Z}^\ell)$ and $R_{E'} : \GAMMA\lra \GL(\mathbb{Z}^{\ell+1})$ be group homomorphisms associated with $\rho$ and $\rho'$ with respect to $E$ and $E'$, respectively.
	Then
	\begin{align}
		R_{E'}(\gamma) = \begin{pmatrix}
			R_E(\gamma) & O\\
			O & 1
		\end{pmatrix}
	\end{align}
	for all $\gamma \in \GAMMA$.
	Since an $((\ell + 1) \times (\ell + 1))$-matrix $T' \ceq \diag(t^{d(P)},\ldots,t^{d(P)})$ commutes with every matrix, we have
	\begin{align}
		\det(I - T'R'_{V}(\gamma)) = \det(I - T'R_{E'}(\gamma)) = (1-t^{d(P)})\det(I - TR_E(\gamma)) = (1 - t^{d(P)})\det(I  - t^{d(P)}\rho(\gamma)),
	\end{align}
	where $T \ceq \diag(t^{d(P)},\ldots.t^{d(P)})$ is an $(\ell \times \ell)$-matrix.
\end{proof}

\begin{example}
	Let $L = \mathbb{Z}e_1 \oplus \mathbb{Z}e_2$ and $P \subseteq L_\mathbb{R}$ be a rational simplex with vertices
	\begin{align}
		v_0 \ceq -e_1,\quad v_1 \ceq e_1,\quad v_2 \ceq \frac{1}{2}e_2.
	\end{align}
	Let $\GAMMA = \{1,\gamma\} \cong \mathbb{Z}/2\mathbb{Z}$ and $\rho : \GAMMA \lra \GL(L)$ be a $\GAMMA$-action defined as
	\begin{align}
		\rho(\gamma)(e_1) = -e_1,\qquad \rho(\gamma)(e_2) = e_2.
	\end{align}
	Then $P$ is $\GAMMA$-invariant, and we have
	\begin{align}
		R_V'(\gamma) = \begin{pmatrix}
			0 & 1 & 0\\
			1 & 0 & 0\\
			0 & 0 & 1
		\end{pmatrix}.
	\end{align}

	Let $d^{(1)} \ceq (1,1,2)$ and $T_1' \ceq \diag(t,t,t^2)$.
	Then 
	\begin{align}
		B(P,d^{(1)})_0 = \{0\},\qquad
		B(P,d^{(1)})_1 = \left\{
			\frac{1}{2}v_0' + \frac{1}{2}v_1'
		\right\},
	\end{align}
	and hence
	\begin{align}
		\Ehr_\rho(P;t) = \frac{\bm{1} + \bm{1}t}{\det(I - T_1'R'_V)},
	\end{align}
	where $\bm{1}$ is the trivial character of $\GAMMA$.
	
	Let $d^{(2)} \ceq (2,2,2)$ and $T_2' \ceq \diag(t^2,t^2,t^2)$.
	Then 
	\begin{align}
		B(P,d^{(2)})_0 = \{0\},\qquad
		B(P,d^{(2)})_1 = \left\{
		v_0',\ v_1',\ \frac{1}{2}v_0' + \frac{1}{2}v_1'
		\right\},\\
		B(P,d^{(2)})_2 = \left\{
		\frac{3}{2}v_0' + \frac{1}{2}v_1',\ \frac{1}{2}v_0' + \frac{3}{2}v_1',\ v_0' + v_1'
		\right\},\qquad
		B(P,d^{(2)})_3 = \left\{
		\frac{3}{2}v_0' + \frac{3}{2}v_1'
		\right\},
	\end{align}
	and hence
	\begin{align}
		\Ehr_\rho(P;t) = \frac{\bm{1} + (2\bm{1} + \bm{\delta})t + (2\bm{1} + \bm{\delta})t^2 + \bm{1}t^3}{\det(I - T_2'R'_V)},
	\end{align}
	where $\bm{\delta}$ is the unique non-trivial irreducible character of $\GAMMA$ such that $\bm{\delta}(1) = 1$ and $\bm{\delta}(\gamma) = -1$.
\end{example}

\section{Cross-polytopes} \label{sec3}

\subsection{The hyperoctahedral group}\ \label{sec3.1}

\noindent
In this section, we summarize several notations and facts about the hyperoctahedral group.
See \cite{Carter, Mayer, Young} for details.

Let $\ell \in \mathbb{Z}_{>0}$ and $I_\ell \ceq \{\pm1,\ldots,\pm\ell\}$.
Let $\mathfrak{H}_\ell$ be a subgroup of the symmetric group $\mathfrak{S}_{I_\ell}$ of $I_\ell$ defined as
\begin{align}
	\mathfrak{H}_\ell \ceq \bigset{\eta \in \mathfrak{S}_{I_\ell}}{\eta(-i) = -\eta(i) \tforall i \in I_\ell},
\end{align}
which is called the \textbf{hyperoctahedral group} of rank $\ell$.
It is well known that $\mathfrak{H}_\ell$ is isomorphic to the Weyl group of type $B_\ell$ and $C_\ell$, the semidirect product $\mathfrak{S}_2^\ell \rtimes \mathfrak{S}_\ell$, and the wreath product $\mathfrak{S}_2 \wr \mathfrak{S}_\ell$.
Hence the symmetric group $\mathfrak{S}_\ell$ is regarded as a subgroup of $\mathfrak{H}_\ell$.
Furthermore, the group $\mathfrak{H}_\ell$ is isomorphic to the automorphism group of the $\ell$-dimensional cross-polytope.

For $i \in \{1,\ldots,\ell\}$, let $\tau_i$ be the coordinate reflection in the $i$-th coordinate, that is, 
\begin{align}
	\tau_i(j) = \begin{cases*}
		-j 	& if $j = i$ or $j = -i$;\\
		j	& otherwise.	
	\end{cases*}
\end{align}
There are three types of permutations corresponding to reflections of the Weyl group of type $B_\ell$ and $C_\ell$:
\begin{iitemize}
	\item $(i\ j) \quad$ for $i,j \in \{1\ldots,\ell\},\ i \neq j$;
	\item $\tau_i \quad$ for $i \in \{1,\ldots,\ell\}$;
	\item $(i\ j)\tau_i\tau_j \quad $ for $i,j \in \{1,\ldots,\ell\},\ i \neq j$,
\end{iitemize}
where $(i\ j)$ is the permutation of $\mathfrak{S}_\ell$ that switches $i$ and $j$.

An \textbf{integer partition} is a sequence of weakly decreasing positive integers.
For an integer partition $\lambda$, the notation $|\lambda|$ is used for the sum of integers in $\lambda$, and $\ell(\lambda)$ is used for the length of $\lambda$.
An integer partition $\lambda$ is a partition of $n$ if $|\lambda| = n$, denoted by $\lambda \vdash n$.
Note that we consider the empty set $\emptyset$ to be a partition of $0$ of length $\ell(\emptyset) = 0$.
A \textbf{bipartition} is a pair of integer partitions.
For a bipartition $(\lambda,\mu)$, define $|(\lambda,\mu)| \ceq |\lambda| + |\mu|$ and $\ell(\lambda,\mu) \ceq \ell(\lambda) + \ell(\mu)$.
A bipartition $(\lambda,\mu)$ is a bipartition of $n$ if $|(\lambda,\mu)| = n$.

The hyperoctahedral group $\mathfrak{H}_\ell$ has the following two types of cyclic permutations:
\begin{align}
	(i_1 \cdots i_t)_+ = \begin{pmatrix}
		i_1 & i_2 & \cdots & i_{t-1} & i_t\\
		i_2 & i_3 & \cdots & i_t & i_1
	\end{pmatrix},\qquad 
	(i_1 \cdots i_t)_- = \begin{pmatrix}
		i_1 & i_2 & \cdots & i_{t-1} & i_t\\
		i_2 & i_3 & \cdots & i_t & -i_1
	\end{pmatrix},
\end{align}
where $i_1,\ldots,i_t \in I_\ell$ have pairwise different absolute values.
We refer to the former as an \textbf{even cyclic permutation} and the latter as an \textbf{odd cyclic permutation}.
Note that an even cyclic permutation of length $1$ is the identity, and the odd cyclic permutation $(i)_-$ of length $1$ is equal to the permutation $\tau_i$.
Each element $\eta \in \mathfrak{H}_\ell$ can be uniquely decomposed into a product of cyclic permutations, up to the order of the factors.
Let $\lambda$ (resp. $\mu$) be an integer partition given by the lengths of even (resp. odd) cyclic permutations in such a cycle decomposition of $\eta$.
The bipartition $(\lambda,\mu)$ is called a \textbf{cycle type} of $\eta$.
Two elements of $\mathfrak{H}_\ell$ are conjugate if and only if they have the same cycle type.
Hence conjugacy classes of $\mathfrak{H}_\ell$ are parameterized by bipartitions of $\ell$.
Note that an element of $\mathfrak{H}_\ell$ with cycle type $(\lambda,\emptyset)$ is conjugate to some element of the subgroup $\mathfrak{S}_\ell$ with cycle type $\lambda$.

Let $(\lambda,\mu)$ be a bipartition of $\ell$.
The \textbf{standard} element for $(\lambda,\mu)$ is an element $\eta \in \mathfrak{H}_\ell$ defined as
\begin{align}
	\eta(i) = \begin{cases*}
		i+1	& if $p_{k-1} < i < p_k$ for some $k \in\{1,\ldots,\ell(\lambda,\mu)\}$;\\
		p_{k-1} + 1			& if $i = p_{k}$ for some $k \in \{1,\ldots,\ell(\lambda)\}$;\\
		-(p_{k-1} + 1)			& if $i = p_{k}$ for some $k \in \{\ell(\lambda) + 1,\ldots,\ell(\lambda,\mu)\}$
	\end{cases*}
\end{align}
for $i \in \{1,\ldots,\ell\}$, where $(p_0,\, p_1,\, \ldots,\, p_{\ell(\lambda,\mu)})$ is the prefix sum of $(\lambda,\mu)$, that is, 
\begin{align}
	p_k = \begin{cases*}
		0												& if $k = 0$;\\
		\sum_{i = 1}^k\lambda_i							& if $k \in \{1,\ldots,\ell(\lambda)\}$;\\
		|\lambda| + \sum_{i=1}^{k - \ell(\lambda)}\mu_i	& if $k \in \{\ell(\lambda) + 1,\ldots,\ell(\lambda,\mu)\}$.
	\end{cases*}
\end{align}
For example, if $\ell= 8$ and $(\lambda,\mu) = ((2,1),(3,2))$, then the standard element for $(\lambda,\mu)$ is 
\begin{align}
	(1\ 2)_+(3)_+(4\ 5\ 6)_-(7\ 8)_- = \begin{pmatrix}
		1 & 2 & 3 & 4 & 5 & 6 & 7 & 8\\
		2 & 1 & 3 & 5 & 6 & -4 & 8 & -7
	\end{pmatrix}.
\end{align}

Furthermore, irreducible characters of $\mathfrak{H}_\ell$ are also parameterized by bipartitions of $\ell$.
The irreducible character $\chi^{(\lambda,\emptyset)}$ of $\mathfrak{H}_\ell$ corresponding to $(\lambda,\emptyset)$ is induced from the irreducible character $\chi^\lambda$ of $\mathfrak{S}_\ell$ corresponding to $\lambda$.

\subsection{Rational cross-polytopes}\ 

\noindent
Let $\ell, s \in \mathbb{Z}_{>0}$ and $d = (d_1,\ldots,d_s) \in (\mathbb{Z}_{>0})^s$.
Let $m = (m_1,\ldots,m_s) \in (\mathbb{Z}_{>0})^s$ satisfy $\ell = m_1 + \cdots + m_s$.
Suppose that $L$ is a lattice of rank $\ell$ with a $\mathbb{Z}$-basis
\begin{align}
	B \ceq \{
		e_{1,1},\ \ldots,\ e_{1,m_1},\ e_{2,1},\ \ldots,\ e_{2,m_2},\ \ldots\ldots\ldots,\ e_{s,1},\ \ldots,\ e_{s,m_s}
	\}
\end{align}
We consider an $\ell$-dimensional rational polytope 
\begin{align}
	P(d,m) \ceq \conv\Bigset{\pm \frac{1}{d_i}e_{i,j}}{i \in \{1,\ldots,s\},\ j \in \{1,\ldots,m_i\}} \subseteq L_\mathbb{R}.
\end{align}
If $s = 1$ and $d = (d_1) = (1)$, then $P(d,m)$ is an $m_1$-dimensional \textbf{cross-polytope}.
Note that Clarke--Higashitani--K\"olbl consider the case where $s = 2$, $d = (1,2)$ and $m = (\ell-1,1)$ in \cite[\S4]{ClarkeHigashitaniKolbl}.

\begin{proposition}\label{prop3.1}
	\begin{align}
		\Ehr(P(d,m);t) = \frac{(1 + t^{d_1})^{m_1} \cdots (1 + t^{d_s})^{m_s}}{(1-t)(1 - t^{d_1})^{m_1} \cdots (1 - t^{d_s})^{m_s}}.
	\end{align}
\end{proposition}
\begin{proof}
	For $n \in \mathbb{Z}_{\geq 0}$, define
	\begin{align}
		X_n \ceq \#\Bigset{(x_{ij})_{1 \leq i \leq s,\, 1 \leq j \leq m_i} \in \mathbb{Z}^\ell}{\sum_{i=1}^sd_i\sum_{j=1}^{m_i}|x_{ij}| = n}.
	\end{align}
	Then
	\begin{align}
		\Ell_{P(d,m)}(q) = \#\Bigset{(x_{ij})_{1 \leq i \leq s,\, 1 \leq j \leq m_i} \in \mathbb{Z}^\ell}{\sum_{i=1}^sd_i\sum_{j=1}^{m_i}|x_{ij}| \leq q} = \sum_{n=0}^q X_n.
	\end{align}
	Hence
	\begin{align}
		\Ehr(P(d,m);t)
		&= \sum_{q \in \mathbb{Z}_{\geq 0}} \sum_{n = 0}^q X_n t^q\\
		&= \left( \sum_{q \in \mathbb{Z}_{\geq 0}} t^q \right) \left( \sum_{n \in \mathbb{Z}_{\geq 0}}X_n t^n \right) 			\\
		&= \left( \sum_{q \in \mathbb{Z}_{\geq 0}} t^q \right) \prod_{i=1}^s \left( \sum_{x \in \mathbb{Z}}t^{d_i|x|} \right)^{m_i}	\\
		&= \frac{(1 + t^{d_1})^{m_1} \cdots (1 + t^{d_\ell})^{m_s}}{(1-t)(1 - t^{d_1})^{m_1} \cdots (1 - t^{d_\ell})^{m_s}}. \qedend
	\end{align}
\end{proof}

Suppose that $d_1,\ldots,d_s$ are distinct.
For each $i \in \{1,\ldots,s\}$, the hyperoctahedral group $\mathfrak{H}_{m_i}$ acts on $\{\pm e_{i,1},\ldots, \pm e_{i,m_i}\}$ as 
\begin{align}
	\eta : e_{i,j} \lmapsto  \begin{cases*}
		e_{i,\, \eta(j)}		& if $\eta(j) > 0$;\\
		-e_{i,\, -\eta(j)}	& if $\eta(j) < 0$
	\end{cases*}
	\quad\tfor j \in \{1,\ldots,m_i\}.
\end{align}
Hence $\GAMMA \ceq \mathfrak{H}_{m_1} \times \cdots \times \mathfrak{H}_{m_s}$ acts linearly on $L$, and let $\rho : \GAMMA \lra \GL(L)$ denote the corresponding representation.
The group $\GAMMA$ is the automorphism group of $P(d,m)$.

\begin{lemma}\label{lem3.2}
	Suppose that $d_1,\ldots,d_s$ are distinct.
	Let $\eta = \eta_1 \cdots \eta_s \in \GAMMA$ ($\eta_i \in \mathfrak{H}_{m_i}$) and $(\lambda^{(i)},\mu^{(i)})$ denote the cycle type of $\eta_i$.
	Define 
	\begin{align}
		d(\eta) \ceq \left(
			\lambda^{(1)}_1d_1,\ \ldots,\ \lambda^{(1)}_{\ell(\lambda^{(1)})}d_1,\ \ldots\ldots\ldots,\ \lambda^{(s)}_1d_s,\ \ldots,\ \lambda^{(s)}_{\ell(\lambda^{(s)})}d_s 
		\right),\qquad
		1 \ceq (1,\ldots,1).
	\end{align}
	Then we have
	\begin{align}
		P(d,m)^\eta \cong P(d(\eta),1).
	\end{align}
\end{lemma}
\begin{proof}
	We can assume that each $\eta_i$ is the standard element for $(\lambda^{(i)},\mu^{(i)})$.
	Then we have
	\begin{align}
		P(d,m)^\eta = \conv\Bigset{\pm\frac{1}{\lambda^{(i)}_k d_i}\Bigl(e_{i,\,p^{(i)}_{k-1} + 1} + \cdots + e_{i,\, p^{(i)}_k}\Bigr)}{i \in \{1,\ldots,s\},\ \ k \in \{1,\ldots,\ell(\lambda^{(i)})\}}, 
	\end{align}
	where $(p^{(i)}_0,\ldots,p^{(i)}_{\ell(\lambda^{(i)})})$ is the prefix sum of $\lambda^{(i)}$.
\end{proof}

It follows from \cref{prop3.1} and \cref{lem3.2} that
\begin{align}
	\Ehr_\rho(P(d,m);t)(\eta) 
	= \sum_{q \in \mathbb{Z}_{\geq 0}}\chi_{P(d,m),q}(\eta) t^q 
	= \Ehr(P(d(\eta),1);t)
	= \frac{ \prod_{i=1}^s \prod_{k=1}^{\ell(\lambda^{(i)})} (1 + t^{\lambda_k^{(i)}d_i}) }{(1-t) \prod_{i=1}^s \prod_{k=1}^{\ell(\lambda^{(i)})} (1 - t^{\lambda_k^{(i)}d_i}) }. \label{s7}
\end{align}

Define an $\ell \times \ell$ diagonal matrix $T$ by
\begin{align}
	T \ceq \diag\Bigl(\, \overbrace{t^{d_1},\ \ldots,\ t^{d_1}}^{m_1},\ \ldots\ldots\ldots,\  \overbrace{t^{d_s},\ \ldots,\ t^{d_s}}^{m_s}\, \Bigr).
\end{align}
Let $R : \GAMMA \lra \GL(\mathbb{Z}^\ell)$ be the homomorphism associated with $\rho$ with respect to the basis $B$.
Then we have
\begin{align}
	\det(I - TR(\eta)) = \prod_{i=1}^s \left(\prod_{k=1}^{\ell(\lambda^{(i)})} (1 - t^{\lambda_k^{(i)}d_i}) \prod_{k=1}^{\ell(\mu^{(i)})} (1 + t^{\mu_k^{(i)}d_i}) \right). \label{s8}
\end{align}

For $i \in \{1,\ldots,s\}$, define $\varphi^{(i)}[t] = \sum_{n \in \mathbb{Z}_{\geq 0}}\varphi_n^{(i)} t^n \in \Cl(\GAMMA)[[t]]$ by
\begin{align}
	\sum_{n \in \mathbb{Z}_{\geq 0}}\varphi^{(i)}_n(\eta) t^n = \prod_{k = 1}^{\ell(\lambda^{(i)})} (1 + t^{\lambda^{(i)}_kd_i})	
	\prod_{k=1}^{\ell(\mu^{(i)})}(1 + t^{\mu^{(i)}_kd_i})
	\tforall \eta \in \GAMMA.	\label{s9}
\end{align}
Thus $\varphi^{(i)}[t]$ is a polynomial of degree $m_id_i$.
Let $\varphi[t] = \sum_{n \in \mathbb{Z}_{\geq 0}}\varphi_nt^n \in \Cl(\GAMMA)[[t]]$ be defined as
\begin{align}
	\varphi[t] = \prod_{i=1}^s\varphi^{(i)}[t].
\end{align}
Then \cref{s7} and \cref{s8} imply
\begin{align}
	\Ehr_\rho(P(d,m);t)
	&= \frac{ \varphi[t] }{(1-t) \det(I - T R) }. \label{EhrCP}
\end{align}

\begin{theorem}\label{thm3.3}
	Each $\varphi_n^{(i)} \in \Cl(\GAMMA)$ is a character of $\GAMMA$.
	Hence $\varphi[t]$ is a polynomial and effective.
	Furthermore, $\varphi[t]$ is palindromic.
	In other words, 
	\begin{align}
		\varphi_n = \varphi_{\deg{\varphi}-n}
	\end{align}
	for all $n \in \{0,\ldots,\deg{\varphi}\}$.
\end{theorem}
\begin{proof}
	We can assume that $d_i = 1$.
	Let $n \in \mathbb{Z}$ satisfy $0 \leq n \leq m_i$.
	We can assume that $\eta_i \in \mathfrak{H}_{m_i}$ is the standard element for $(\lambda^{(i)},\mu^{(i)})$.
	Define a set $X(n,\eta_i)$ by 
	\begin{align}
		X(n,\eta_i) \ceq \Bigset{(K_1,K_2)}{K_1 \subseteq \{1,\ldots,\ell(\lambda^{(i)})\},\ \ K_2 \subseteq \{1,\ldots,\ell(\mu^{(i)})\},\ \ \sum_{k \in K_1}\lambda^{(i)}_k + \sum_{k \in K_2}\mu^{(i)}_k = n}.
	\end{align}
	Then $\#X(n,\eta_i)$ is equal to the coefficient of $t^n$ in the polynomial in \cref{s9}.
	Hence
	\begin{align}
		\varphi^{(i)}_n(\eta_i) = \#X(n,\eta_i).
	\end{align}
	Let $T(n)$ be a set of subsets of $[m_i] \ceq \{1,\ldots,m_i\}$ of cardinality $n$:
	\begin{align}
		T(n) \ceq \bigset{I \subseteq [m_i]}{\#I = n}. \label{s10}
	\end{align}
	The hyperoctahedral group $\mathfrak{H}_{m_i}$ acts on $T(n)$ as follows:
	\begin{align}
		\eta : I = \{i_1,\ldots,i_n\} \lmapsto \bigl\{\, |\eta(i_1)|,\, \ldots,\, |\eta(i_n)| \, \bigr\} \quad \tfor \eta \in \mathfrak{H}_{m_i}.
	\end{align}
	Then there exists a bijection $\PHI : X(n,\eta_i) \lra T(n)^{\eta_i} \ceq \tbigset{I \in T(n)}{\eta_i(I) = I}$ defined as 
	\begin{align}
		\PHI(K_1,K_2) = \bigsqcup_{k \in K_1}\bigset{j \in [m_i]}{p_{k-1}^{(i)} < j \leq p_k^{(i)}} \sqcup \bigsqcup_{k \in K_2}\bigset{j \in [m_i]}{p^{(i)}_{\ell(\lambda^{(i)}) + k-1} < j \leq p^{(i)}_{\ell(\lambda^{(i)}) + k}}.
	\end{align}
	Therefore 
	\begin{align}
		\varphi_n^{(i)}(\eta_i) = \#T(n)^{\eta_i},
	\end{align}
	and hence $\varphi_n^{(i)}$ coincides with the character of the permutation representation of $\mathfrak{H}_{m_i}$ on $T(n)$.
	
	Let $n \in \mathbb{Z}_{\geq 0}$.
	The coefficient of $t^n$ in $\varphi[t]$ is equal to 
	\begin{align}
		\varphi_n = \sum_{\substack{n_1,\ldots,n_s \in \mathbb{Z}_{\geq 0}\\n_1 + \cdots + n_s = n}}\left(\prod_{i=1}^s \varphi^{(i)}_{n_i} \right).
	\end{align}
	Since $\varphi^{(i)}_{n_i}$ is a character of $\mathfrak{H}_{m_i}$ for each $i \in \{1,\ldots,s\}$, the product $\prod_{i=1}^s\varphi^{(i)}_{n_i}$ can be regarded as a character of $\GAMMA$.
	Hence $\varphi[t]$ is effective.

	We will prove that each $\varphi^{(i)}[t]$ is palindromic.
	Let $n \in \{0,\ldots,\deg{\varphi^{(i)}}\}$.
	Then the map $\kappa : T(n) \lra T(\deg{\varphi^{(i)}}-n)$ defined as
	\begin{align}
		\kappa : I \lmapsto [m_i]\setminus I
	\end{align}
	is a bijection such that
	\begin{align}
		\kappa(\eta(I)) = \eta(\kappa(I)) \quad \tfor \eta \in \mathfrak{H}_{m_i}.
	\end{align}
	Hence $\varphi_n^{(i)}(\eta) = \#T(n)^\eta = \#T(\deg{\varphi^{(i)}}-n)^\eta = \varphi_{\deg{\varphi^{(i)}}-n}^{(i)}(\eta)$ for any $\eta \in \mathfrak{H}_{m_i}$
\end{proof}
	
\begin{remark}\label{rmk4}
	We again assume that $d_i = 1$.
	Let $n \in \mathbb{Z}$ satisfy $0 \leq n \leq m_i$.
	Consider a Young diagram $Y$ given by the integer partition $\nu$ defined as
	\begin{align}
		\nu \ceq \begin{cases*}
			(m_i-n,\, n)	& if $n \leq \frac{m_i}{2}$;\\
			(n,\, m_i-n)	& if $n > \frac{m_i}{2}$.
		\end{cases*}
	\end{align}
	A numbering of $Y$ is an assignment of the integers $\{1,\ldots,m_i\}$ to the boxes of $Y$ such that all the entries are distinct.
	A tabloid is an equivalence class of numberings of $Y$, where two numberings are equivalent if the corresponding rows contain the same entries.
	Let $\mathcal{T}(\nu)$ denote the set of all tabloids of $Y$.
	We can see that there exists a bijection between $\mathcal{T}(\nu)$ and $T(n)$
	as $\mathfrak{H}_{m_i}$-sets.
	Let $M^\nu$ denote the $\mathbb{C}$-linear space generated by $\mathcal{T}(\nu)$, which plays an important role in the representation theory of the symmetric group $\mathfrak{S}_{m_i}$.
	Then $\varphi^{(i)}_n$ is equal to the character of the representation $M^\nu$.
	In particular, it is shown in \cite[\S7.2 (6)]{Fulton} that the character of $M^\nu$ is equal to the induced character of the trivial character from $\mathfrak{S}_{m_i-n} \times \mathfrak{S}_n$ to $\mathfrak{S}_{m_i}$.
	Furthermore, $M^\nu$ can be decomposed into a direct sum of Specht modules as a $\mathbb{C}[\mathfrak{S}_{m_i}]$-module as follows \cite[\S7.2 Corollary 1]{Fulton}:
	\begin{align}
		M^\nu \cong \bigoplus_{\substack{\lambda \vdash m_i\\\lambda \trianglerighteq \nu}}(S^\lambda)^{\oplus K_{\lambda,\nu}}, \label{s11}
	\end{align}
	where $K_{\lambda,\nu}$ is the Kostka number, and $\lambda \trianglerighteq \nu$ if and only if
	\begin{align}
		\lambda_1 + \cdots + \lambda_k \geq \nu_1 + \cdots + \nu_k \tforall k.
	\end{align}
	Moreover, \cref{s11} also gives an irreducible decomposition as a $\mathbb{C}[\mathfrak{H}_{m_i}]$-module.
	Hence we have
	\begin{align}
		\varphi_n^{(i)} = \sum_{\substack{\lambda \vdash m_i\\ \lambda \trianglerighteq \nu}}K_{\lambda,\nu}\chi^{(\lambda,\emptyset)}.
	\end{align}
\end{remark}

For $i \in \{1,\ldots,s\}$, let $N_{m_i}$ be a subgroup of $\mathfrak{H}_{m_i}$ generated by the coordinate reflections.
Let $N \ceq N_{m_1} \times \cdots \times N_{m_s} \subseteq \GAMMA$, that is, 
\begin{align}
	N = \Bigset{\tau_J \in \GAMMA}{J \subseteq \bigset{(i,j)}{i \in \{1,\ldots,s\},\ j \in \{1,\ldots,m_i\}}},
\end{align}
where $\tau_J$ is defined as
\begin{align}
	\tau_J : e_{i,j} \lmapsto \begin{cases*}
		-e_{i,j}	& if $(i,j) \in J$;\\
		e_{i,j}		& if $(i,j) \not\in J$.
	\end{cases*}
\end{align}

Let $\rho|_N$ and $R|_N$ be the restrictions of $\rho$ and $R$ on $N$, respectively.

\begin{theorem}\label{thm4.6}
	
	\begin{align}
		\Ehr_{\rho|_N}(P(d,m);t)
		= \frac{(1 + t^{d_1})^{m_1} \cdots (1 + t^{d_s})^{m_s} \cdot \bm{1}}{(1-t)\det(I - TR|_N)},	\label{s6}
	\end{align}
	where $\bm{1}$ is the trivial character of $N$.
\end{theorem}
\begin{proof}
	Although the proof can be obtained as a corollary of \cref{thm3.3}, we will present a direct proof here.

	For each $J \subseteq \tbigset{(i,j)}{i \in \{1,\ldots,s\},\ j \in \{1,\ldots,m_i\}}$, we can see that
	\begin{align}
		\det(I - TR|_N(\tau_J)) = \prod_{(i,j) \in J}(1 + t^{d_i}) \prod_{(i,j) \not\in J}(1 - t^{d_i}),
	\end{align}
	and 
	\begin{align}
		P(d,m)^{\tau_J} = \conv\Bigset{\pm \frac{1}{d_i}e_{i,j}}{(i,j) \not\in J}.
	\end{align}
	Thus, by \cref{prop3.1},
	\begin{align}
		\Ehr_{\rho|_N}(P(d,m);t)(\tau_J)
		&= \sum_{q \in \mathbb{Z}_{\geq 0}}\chi_{P(d,m),q}(\tau_J) t^q \\
		&= \frac{\prod_{(i,j) \not\in J}(1 + t^{d_i})}{(1 - t)\prod_{(i,j) \not\in J}(1 - t^{d_i})} \\
		&= \frac{(1 + t^{d_1})^{m_1} \cdots (1 + t^{d_s})^{m_s}}{(1-t)\det(I - TR|_N(\tau_J))}.
	\end{align}
	This equation implies \cref{s6}.
\end{proof}

\begin{example}\label{ex_CP1}
We consider a polytope $P$ given by 
\begin{align}
	P = P((2,4),(2,1)) = \conv\left\{
		\pm\frac{1}{2}e_1,\ \pm\frac{1}{2}e_2,\ \pm\frac{1}{4}e_3
	\right\} \subseteq \mathbb{R}^3.
\end{align}
Then the automorphism group $\GAMMA$ is equal to $\mathfrak{H}_2 \times\mathfrak{H}_1$.
Let $R : \GAMMA\lra \GL(L)$ be the homomorphism associated with the $\GAMMA$-action on $\mathbb{Z}^3$.
The character table of $\GAMMA$ is shown in \cref{table2}.

\begin{table}[h]
	\caption{The character table of $\mathfrak{H}_2 \times \mathfrak{H}_1$} \label{table2}
	\scalebox{0.9}{
	\begin{tabular}{l|rrrrrrrrrr}
		& \ \ 1 & $(1\ 2)_+$ & $(1)_-$ & $(1)_-(2)_-$ & $(1\ 2)_-$ & $(3)_-$ & $(1\ 2)_+(3)_-$ & $(1)_-(3)_-$ & $(1)_-(2)_-(3)_-$ & $(1\ 2)_-(3)_-$ \\ \hline
		$\bm{1} = \chi^{((2),\emptyset)}\chi^{((1),\emptyset)}$       & 1 & 1          & 1       & 1            & 1          & 1       & 1               & 1            & 1                 & 1               \\
		$\chi_1 = \chi^{(\emptyset,(1^2))}\chi^{((1),\emptyset)}$     & 1 & -1         & -1      & 1            & 1          & 1       & -1              & -1           & 1                 & 1               \\
		$\chi_2 = \chi^{(\emptyset,(2))}\chi^{((1),\emptyset)}$     & 1 & 1          & -1      & 1            & -1         & 1       & 1               & -1           & 1                 & -1              \\
		$\chi_3 = \chi^{((1^2),\emptyset)}\chi^{((1),\emptyset)}$     & 1 & -1         & 1       & 1            & -1         & 1       & -1              & 1            & 1                 & -1              \\
		$\chi_4 = \chi^{((1),(1))}\chi^{((1),\emptyset)}$     & 2 & 0          & 0       & -2           & 0          & 2       & 0               & 0            & -2                & 0               \\
		$\chi_5 = \chi^{((2),\emptyset)}\chi^{(\emptyset,(1))}$         & 1 & 1          & 1       & 1            & 1          & -1      & -1              & -1           & -1                & -1              \\
		$\chi_6 = \chi^{(\emptyset,(1^2))}\chi^{(\emptyset,(1))}$ & 1 & -1         & -1      & 1            & 1          & -1      & 1               & 1            & -1                & -1              \\
		$\chi_7 = \chi^{(\emptyset,(2))}\chi^{(\emptyset,(1))}$ & 1 & 1          & -1      & 1            & -1         & -1      & -1              & 1            & -1                & 1               \\
		$\chi_8 = \chi^{((1^2),\emptyset)}\chi^{(\emptyset,(1))}$ & 1 & -1         & 1       & 1            & -1         & -1      & 1               & -1           & -1                & 1               \\
		$\chi_9 = \chi^{((1),(1))}\chi^{(\emptyset,(1))}$ & 2 & 0          & 0       & -2           & 0          & -2      & 0               & 0            & 2                 & 0              
	\end{tabular}
}
\end{table}

Let $T_1 = \diag(t^2,t^2,t^4)$.
Then the numerator $\varphi[t]$ of \cref{EhrCP} is obtained as
\begin{align}
	\varphi[t] = \left(\sum_{n \in \mathbb{Z}_{\geq 0}}\varphi_n^{(1)}t^n\right)\left(\sum_{n \in \mathbb{Z}_{\geq 0}}\varphi_n^{(2)}t^n\right).	
\end{align}
According to the observations in \cref{rmk4}, $\varphi_n^{(1)}$ and $\varphi_n^{(2)}$ are given as
\begin{align}
	\varphi_n^{(1)} = \begin{cases*}
		\chi^{((2),\emptyset)} 		& if $n = 0 \tor 4$;\\
		\chi^{((1^2),\emptyset)} + \chi^{((2),\emptyset)}		& if $n = 2$;\\
		0		& otherwise,
	\end{cases*}\qquad 
	\varphi_n^{(2)} = \begin{cases*}
		\chi^{((1),\emptyset)} 		& if $n = 0 \tor 4$;\\
		0		& otherwise.
	\end{cases*}
\end{align}
Hence
\begin{align}
	\varphi[t] &= \left( \chi^{((2),\emptyset)} + (\chi^{((1^2),\emptyset)} + \chi^{((2),\emptyset)})t^2 + \chi^{((2),\emptyset)}t^4 \right)\left( \chi^{((1),\emptyset)} + \chi^{((1),\emptyset)}t^4 \right)\\
	&= \bm{1} + (\bm{1} + \chi_3)t^2 + 2\bm{1}t^4 + (\bm{1} + \chi_3)t^6 + \bm{1}t^8.
\end{align}

\end{example}

\section*{Acknowledgment}

The author would like to thank Professor Masahiko Yoshinaga and Professor Akihiro Higashitani for helpful discussions and comments on this research. 
The author also acknowledges support by JSPS KAKENHI, Grant Number 25KJ1735.

%%%%%%%%%%%%%%%%%%%%%%%%%%%%%%%%%%%%%%%%%%%%%%%

\bibliographystyle{amsplain}
\bibliography{bibfile}

\end{document}